\documentclass[12pt,oneside,english]{amsart}
\usepackage[T1]{fontenc}
\usepackage[latin9]{inputenc}
\usepackage{babel}
\usepackage{float}
\usepackage{textcomp}
\usepackage{amstext}
\usepackage{amsthm}
\usepackage{amssymb}
\usepackage{graphicx}
\usepackage{setspace}
\usepackage[unicode=true,pdfusetitle,
 bookmarks=true,bookmarksnumbered=false,bookmarksopen=false,
 breaklinks=false,pdfborder={0 0 1},backref=false,colorlinks=false]
 {hyperref}

\makeatletter

\providecommand{\tabularnewline}{\\}

\numberwithin{equation}{section}
\numberwithin{figure}{section}
\theoremstyle{plain}
\newtheorem{thm}{\protect\theoremname}
\theoremstyle{plain}
\newtheorem{prop}[thm]{\protect\propositionname}
\theoremstyle{remark}
\newtheorem*{rem*}{\protect\remarkname}

\providecommand{\theoremname}{Theorem}

\makeatother

\providecommand{\propositionname}{Proposition}
\providecommand{\remarkname}{Remark}
\providecommand{\theoremname}{Theorem}

\begin{document}
\title{Two-Sample Test Based on Classification Probability}
\author{Haiyan Cai, Bryan Goggin, Qingtang Jiang}
\address{Department of Mathematics and Computer Science, University of Missouri
- St. Louis, One University Boulevard, St. Louis, Missouri, 63121}
\begin{abstract}
Robust classification algorithms have been developed in recent years
with great success. We take advantage of this development and recast
the classical two-sample test problem in the framework of classification.
Based on the estimates of classification probabilities from a classifier
trained from the samples, a test statistic is proposed. We explain
why such a test can be a powerful test and compare its performance
in terms of the power and efficiency with those of some other recently
proposed tests with simulation and real-life data. The test proposed
is nonparametric and can be applied to complex and high dimensional
data wherever there is a classifier that provides consistent estimate
of the classification probability for such data.
\end{abstract}

\keywords{Two-sample test, classification.}

\maketitle
\newpage

\section{Introduction}

A two-sample test detects if two sets of data have different underlying
probability distributions. This is a classical problem in statistics
with numerous applications in a range of scientific inquiries. The
problem becomes challenging when data have higher dimensions relative
to their sizes and are in complex form (texts, curves, images, graphs
etc.). Classical methods with a focus on testing the difference in
the first or second moments of the distributions have become inadequate
for such data and remedies have been suggested \cite{Bai-Saranadasa,Cai-Liu-Xia,Chen-Qin,Chakraborty-Chaudhuri,Hu-Bai}.
Other and more general methods have also been proposed. For example,
\cite{Biau-Gyorfi} investigates a test based on the $L_{1}$ distance
between the two empirical distributions, \cite{Gretton-Borgwardt-al}
proposes a test statistic (the maximum mean discrepancy) which is
the largest of the differences in means over functions in the unit
ball of a reproducing kernel Hilbert space and \cite{Chwialkowski-al.}
proposes tests based on an ensemble of distances between analytic
functions representing each of the distributions. How these tests
relate to the most powerful test in the Nayman-Pearson framework is
not clear however.

On the other hand, robust algorithms have been developed recently
for classification problems. These algorithms are capable of training
classifiers to discriminate complex and very high dimensional data
from different distributions. A good classifier learned from training
data can predict with high probability the correct class for a new/testing
data point. It is therefore reasonable to ask if a good classifier
can be adopted for testing the difference between two samples from
different probability distributions \cite{Lopez-Paz-Oquab,Ramdas-Singh-Wasserman}.
One can assign labels 1 or 0 to data points from the two samples.
Then an intuitive approach is to utilize the prediction accuracy of
a classifier trained from the labeled data for the test. Under the
null hypothesis of identical distributions, the prediction accuracy
(the probability of predicting class labels correctly) should not
be too high, and therefore the null hypothesis should be rejected
if otherwise. The versatility of such a test has been demonstrated
in \cite{Lopez-Paz-Oquab} with large and high dimensional datasets.
There are some disadvantages in this approach however. The prediction
accuracy has to be estimated through cross-validation which makes
the test less efficient in data utilization and can slow down the
computation. More importantly, as we will see in the next section,
a more powerful test which is not based on prediction accuracy can
be derived. 

Following a different consideration, we propose a simple yet powerful
method for the two-sample test. By treating the two sample data as
a set of sample feature points with labels of either 1 or 0, depending
on which sample a feature point is from, we can re-formulate the original
null and alternative hypotheses into an equivalent pair of hypotheses
on the joint distribution of the feature points and their labels.
In this setting, the concept of classification probability in classification
problems, the probability that the label of a given feature point
is 1, is naturally induced, and a classification algorithm, like logistic
regression, random forest, $k$-nearest neighbor or support vector
machine, can be used to estimate such a probability from the given
data \cite{GyorfiL,Kruppa-et-al,Malley-el-al,Steinwart-Christmann}.
The nice thing is that the likelihood ratio in the two-sample testing
problem can be expressed exactly as the odds ratio of the classification
probability multiplied by a known constant. This connection allows
us to approximate the likelihood ratio with the odds ratio of the
classification probability to derive a powerful test statistic. Unlike
the test based on classification accuracy which in \cite{Ramdas-Singh-Wasserman}
is considered as an indirect approach to the two-sample problem, our
classification probability test is basically an approximation of the
likelihood ratio test. Therefore asymptotically there should be no
loss of information like the classification accuracy test.

Our approach assumes that we have a consistent estimator of the classification
probability. There are some results in the literature \cite{GyorfiL,Steinwart,Steinwart-Christmann,Wager-Athey},
but more studies are needed in order to identify better theoretical
sufficient conditions and this will not be the focus of the current
paper. Through examples, we will compare our method with the maximum
mean discrepancy (MMD) test \cite{Gretton-Borgwardt-al} and a classifier
test based on classification accuracies \cite{Lopez-Paz-Oquab,Ramdas-Singh-Wasserman}
and show significant improvements in power and efficiency of our test.
We will call our test Classification Probability Test (CP test or
CPT for short).

In the following, we introduce our test in Section 2 and demonstrate
its performance through examples in Section 3. Summary and discussions
are given in Section 4.

\section{Two-sample tests based on classification probability}

Suppose we have independent samples from probability distributions
$F$ and $G$ respectively on some feature space $\mathcal{X}$. For
simplicity, we assume $F$ and $G$ have density functions $f(x)$
and $g(x)$ respectively on $\mathcal{X}$ relative to some reference
measure. Let
\[
\mathcal{S}_{f}=\{X_{f,1},...,X_{f,n}\}\sim f^{n}\text{ and }\mathcal{S}_{g}=\{X_{g,1},...,X_{g,m}\}\sim g^{m}
\]
 be two independent samples. The problem is to test
\begin{equation}
H_{0}:f=g\ \text{ vs. }H_{1}:f\not=g.\label{eq:the test}
\end{equation}

It is possible to reformulate this test within the framework of a
classification problem, as it will become clear below. To this end,
we first pool the samples together and let $\mathcal{S}=\mathcal{S}_{f}\cup\mathcal{S}_{g}$,
$N=n+m$. To each $X_{i}\in\mathcal{S},i=1,...,N,$ we assign a class
variable $Y_{i}\in\{0,1\}$ such that
\[
Y_{i}=\begin{cases}
1 & \text{if }X_{i}\in\mathcal{S}_{f}\\
0 & \text{if }X_{i}\in\mathcal{S}_{g}.
\end{cases}
\]
In this way we obtain an augmented dataset: $(X_{i},Y_{i}),i=1,...,N$.
The data in the new form can be viewed as an i.i.d. samples from the
joint distribution of a pair of random variables $(X,Y)$ for which
$P(Y=1)=\pi=1-P(Y=0)$ for some $\pi\in(0,1)$ and the conditional
distribution of $X$ is $f$ given $Y=1$, and $g$ given $Y=0$.
Let $p(x)=P(Y=1|X=x)$ be the class probability for a given $x\in\mathcal{X}$.
Then, by Bayes Theorem,
\begin{equation}
p(x)=\frac{\pi f(x)}{\pi f(x)+(1-\pi)g(x)}.\label{eq:cond prob}
\end{equation}
Note that $p(x)$ depends on $\pi$. At first glance, it might seem
that we are giving ourself one extra parameter $\pi$ to deal with.
But in reallity, we don't need to know its value, as long as we consider
$n/N$ as its estimate or simply assume $n/N$ is $\pi$.

Next, let us consider the power of a test in testing (\ref{eq:the test})
at a given significance level $\alpha$. Note that the underlying
density functions $f$ and $g$ are unknown but fixed. It turns out
that an upper bound for the power of any test in our two-sample problem
can be expressed in terms of the class probability $p(x)$. To see
this, let
\begin{equation}
U=\frac{1}{n}\sum_{i=1}^{n}\left(\log\frac{p(X_{f,i})}{1-p(X_{f,i})}-\log\frac{\pi}{1-\pi}\right).\label{eq:U}
\end{equation}
Note that the functions $p(x)$ and $f(x)$ have the same support
and $X_{f,i}$ are sampled from $f$, therefore the terms $p(X_{f,i})$
in (\ref{eq:U}) are always strictly positive. If $p(X_{f,i})=1$
for some $i=1,...,n$, we set $U=\infty$. This happens only when
some $X_{f,i}$ assumes a value that is not in the support of the
density function $g(x)$. The following result motivates our test
to be given later. It is essentially a form of the usual Neyman-Pearson
Lemma.
\begin{prop}
\label{prop:prop 1} Let $T$ be any test statistic based on samples
from the probability densities $f$ and g for testing (\ref{eq:the test}).
For every $\alpha\in(0,1)$, let $C_{T}\subset\mathbb{R}$ be the
critical region of $T$ such that $P(T\in C_{T})=\alpha$ under $H_{0}$.
Then the power of $T$ has the bound:
\begin{equation}
P(T\in C_{T})\le P(U>c_{U}),\label{eq:bound for power}
\end{equation}
where $c_{U}\in\mathbb{R}$ is such that
\[
P(U>c_{U})=\alpha\text{ under }H_{0}.
\]
\end{prop}

\begin{proof}
We can restate the hypotheses in (\ref{eq:the test}) in a slightly
different but equivalent way. First, let us us assume that sample
$\mathcal{S}_{g}$ is from some unknown density $g$, and $f$ is
another unknown density which is different from $g$:
\[
f\not=g.
\]
Under this assumption, the two-sample problem can then be stated as:
under $H_{0}$, the sample $\mathcal{S}_{f}$ is from the same density
$g$ as $\mathcal{S}_{g}$ and, under $H_{1}$, $\mathcal{S}_{f}$
is from $f$, or
\begin{equation}
H_{0}':\mathcal{S}_{f}\sim g^{n}\ vs.\ H_{1}':S_{f}\sim f^{n}.\label{eq:H'-1}
\end{equation}
Let's adopt this form of the two-sample test in the proof. Next, we
consider the joint density of $(X,Y)$. Under $H'_{1}$, this joint
density has the form
\[
h_{1}(x,y)=(yf(x)+(1-y)g(x))\pi^{y}(1-\pi)^{1-y}.
\]
This $h_{1}$ can be rewritten as
\[
h_{1}(x,y)=[g(x)+y(f(x)-g(x))]\pi^{y}(1-\pi)^{1-y}.
\]
We see that the joint density of $(X,Y)$ under $H_{0}$ takes the
form
\[
h_{0}(x,y)=g(x)\pi^{y}(1-\pi)^{1-y}.
\]
Let $h(x,y)$ denote the generic joint density of $(X,Y)$. We can
rewrite (\ref{eq:H'-1}) as
\begin{equation}
H''_{0}:h=h_{0}\ vs.\ H''_{1}:h=h_{1}.\label{eq:H''}
\end{equation}
The log of the density ratio for testing (\ref{eq:H''}) can be written
as:
\begin{equation}
\log\frac{h_{0}(x,y)}{h_{1}(x,y)}=y\log\frac{g(x)}{f(x)},y\in\{0,1\},x\in\mathcal{X}.\label{eq:ll}
\end{equation}
From (\ref{eq:cond prob}), 
\[
\frac{f(x)}{g(x)}=\frac{(1-\pi)p(x)}{\pi(1-p(x))}
\]
and hence $U$ given in (\ref{eq:U}) can also be written as
\[
U=\frac{1}{n}\sum_{i=1}^{n}\log\frac{f(X_{f,i})}{g(X_{f,i})}=-\frac{1}{n}\sum_{i=1}^{N}Y_{i}\log\frac{g(X_{i})}{f(X_{i})}.
\]
This, plus (\ref{eq:ll}), shows that, up to a constant factor $1/n$,
$U$ is actually the negative log-likelihood ratio of the data from
the joint distribution of $(X,Y)$. Now the standard arguments used
in the proof of the Neyman-Pearson Lemma lead us to the inequality
(\ref{eq:bound for power}).
\end{proof}
\begin{rem*}
(1) Whenever the law of large numbers holds here, the quantity $U$
converges in probability to the Kullback-Liebler divergence from $g$
to $f$:
\[
U\to E_{f}\left(\log\frac{f(X)}{g(X)}\right),\text{ as }n,m\to\infty\text{ and }n/N\to\pi\in(0,1).
\]
Therefore $U$ estimates the K-L distance from $g$ to $f$. (2) We
see in the proof that the quantities $U$ and the critical value $c_{U}$
in (\ref{eq:bound for power}) can be replaced with 
\[
V=\frac{1}{m}\sum_{i=1}^{m}\left(\log\frac{1-p(X_{g,i})}{p(X_{g,i})}-\log\frac{1-\pi}{\pi}\right)
\]
and the critical value $c_{V}$, defined in a similar way as $c_{U}$.
\end{rem*}
Unfortunately, $U$ depends on unknown $f$ and $g$. However, with
the bound given in Proposition \ref{prop:prop 1}, if we can find
a statistic $W$ such that $P(W\not=U)\to0$ asymptotically, then
we can use $W$ to perform a test that is asymptotically most powerful.
Of course, the sampling distribution of $W$ under $H_{0}$ would
be generally unknown. But in many practical problems, one can circumvent
the difficulty with a permutation test. An important feature of $U$
is that we don't need to estimate $f$ and $g$, or the ratio $f/g$,
to calculate $U$, which we know can be increasingly impossible as
the dimensionality of the feature space $\mathcal{X}$ grows. On the
other hand, there are many powerful classification algorithms (random
forests, support vector machines, deep neural networks etc.) working
on large, complex and very high dimensional data that estimate the
class probabilities effectively. These observations motivate the following
test.

Let $\hat{p}(x)$ be a consistent estimate of the classification probability
$p(x)$. We propose the plug-in test statistic

\begin{equation}
W_{1}=\frac{1}{n}\sum_{i=1}^{n}\log\frac{\hat{p}(X_{1i})}{1-\hat{p}(X_{1i})}-\log\frac{n}{m}\label{eq:W}
\end{equation}
for a permutation test. In this test, the corresponding null hypothesis
$H_{0}^{\text{perm}}$ states that in the data, each $(X_{i},Y_{i})$
is from a distribution with $Y_{i}\sim B(1,\pi)$ and $X_{i}\sim\pi f+(1-\pi)g$
independently. Note that the original $H_{0}$ implies $H_{0}^{\text{perm}}$,
but not neccesaryly the other away around. This however will not affect
the validity of the test. Here is the test.

\vskip .1in \noindent {\bf The Classification Probability Test 1.}
Given $\alpha\in(0,1)$.
\begin{enumerate}
\item Generate a number of values of $W_{1}$ using $\hat{p}(x)$ and samples
from the null hypothesis $H_{0}^{\text{perm}}$.
\item Find the critical value $c_{1}$ satisfying (approximately) $P_{H_{0}^{\text{perm}}}(W_{1}>c_{1})=\alpha$
based on the values of $W_{1}$ generated in step (1).
\item Calculate $W_{1}$ from the original data and reject $H_{0}$ if $W_{1}>c_{1}$.\vskip .1in
\end{enumerate}
More specifically, the steps (1) and (2) above is implemented as follows.
First, following the simple random sampling principle, we randomly
divide the pooled sample $\mathcal{S}$ into subsets $\mathcal{S}_{f}^{\text{perm}}$
and $\mathcal{S}_{g}^{\text{perm}}$ of sizes $n$ and $m$ respectively.
We assign label 1 to $X_{i}$'s in $\mathcal{S}_{f}^{\text{perm}}$
and label $0$ to $X_{i}$'s in $\mathcal{S}_{g}^{\text{perm}}$ and
calculate $\hat{p}(x)$ from a given classifier (a support vector
machine classifier or a random forest classifier, for example). A
value of $W_{1}$ in (\ref{eq:W}) based on this shuffled data is
obtained. This computation is repeated independently for a sufficiently
large number of times. The critical value $c$ is then the $\alpha$th
sample percentile based on the values of $W_{1}$. Finally, $W_{1}$
is culculated from the original data using the same $\hat{p}(x)$
and we reject $H_{0}^{\text{perm}}$ whenever $W_{1}>c$.

In the next session we will demonstrate the actual power of this test
with several examples. Gnerally, the performance of the test depends
on the underlying distribution of the data and the method one uses
to obtain $\hat{p}(x)$, but in our experiments, we were always able
to find a classifier so the test based on it out performed all other
tests. The efficiency of the test depends on the rate of convergence
of the estimator $\hat{p}$. To have some idea of theoretical guarantee
of our method, we state the following proposition under a uniform
consistency condition.
\begin{prop}
\label{prop:2}Assume when $n,m\to\infty\text{ and }n/(n+m)\to\pi\in(0,1)$,
\begin{equation}
\sup_{x\in\mathcal{X}}\left(\hat{p}(x)-p(x)\right)\to0\ \text{ in probability}.\label{eq:ucc}
\end{equation}
 Then the test given above is an asymptotically most powerful test.
\end{prop}

\begin{proof}
Here is a sketch of the proof. The condition (\ref{eq:ucc}) implies
that for any small $\delta>0$
\[
\sup_{x:p(x)\in[\delta,1-\delta]}\left|\left(\log\frac{\hat{p}(x)}{1-\hat{p}(x)}-\log\frac{n}{m}\right)-\left(\log\frac{p(x)}{1-p(x)}-\log\frac{\pi}{1-\pi}\right)\right|\to0\text{ in probability},
\]
due to the uniform continuity of the function $\log u/(1-u)$ on the
interval $[\delta,1-\delta]$. On the other hand, if $p(x)=0\text{ or }1$
then (\ref{eq:ucc}) implies that $\log\hat{p}(x)/(1-\hat{p}(x))\to-\infty\text{ or }\infty$
respectively. From this it can be shown that
\[
|W_{1}-U|\to0\ \text{ in probability as }n,m\to\infty\text{ and }n/(n+m)\to\pi\in(0,1),
\]
or
\begin{equation}
P(W_{1}\not=U)\to0\ \text{ as }n,m\to\infty\text{ and }n/(n+m)\to\pi\in(0,1).\label{eq:W1=00003DU}
\end{equation}
Hence
\begin{align*}
\left|P(W_{1}>c)-P(U>c)\right| & \le P(W_{1}\not=U)\to0.
\end{align*}
Therefore the power of the test based on $W_{1}$ asymtoptically achieves
the upper bound given in Proposition \ref{prop:prop 1}.
\end{proof}
\begin{rem*}
The uniform consistency condition (\ref{eq:ucc}) is somewhat strong
and artificial. More studies are needed to identify weaker conditions
for (\ref{eq:W1=00003DU}) to hold.
\end{rem*}
Next, we point out that it is possible to propose other tests using
$\hat{p}(x)$ based on more heuristic arguments. We give one here.
The form of the conditional probability in (\ref{eq:cond prob}) suggests
that the two-sample test is also equivalent to determining if the
mapping $x\to p(x)$ is a constant function. For if $p(x)$ is constant
for all $x$, then $p(x)=E_{X,Y}p(X)$ for all $x$. But $E_{X,Y}p(X)=\pi$.
Thus $p(x)=\pi$ for all $x$ and this condition forces $f\equiv g$
by (\ref{eq:cond prob}). Now to determine if $p(x)$ is a constant
function, we can consider the variance
\[
\theta=Var_{X,Y}(p(X))=E_{X,Y}p(X)^{2}-\pi^{2}.
\]
The hypotheses in (\ref{eq:the test}) can now be restated as
\begin{equation}
H'''_{0}:\theta=0\ \text{vs.}\ H'''_{1}:\theta>0.\label{eq:H0vsH1}
\end{equation}
A natural test statistic for this test can then be
\begin{equation}
W_{2}=\frac{1}{N}\sum_{i=1}^{N}\left(\hat{p}(X_{i})-\bar{p}\right)^{2},\label{eq:W1}
\end{equation}
where $\bar{p}=\frac{1}{N}\sum_{i=1}^{N}p(X_{i})$. Here is the permutation
test based on $W_{2}$:

\vskip .1in \noindent {\bf The Classification Probability Test 2.}
Given $\alpha\in(0,1)$.
\begin{enumerate}
\item Independently, generate a number of values of $W_{2}$ using $\hat{p}(x)$
and the samples from the null hypothesis $H_{0}^{\text{perm}}$.
\item Find the critical value $c_{2}$ satisfying (approximately) $P_{H_{0}^{\text{perm}}}(W_{2}>c_{2})=\alpha$
based on the values of $W_{2}$ generated in step (1).
\item Calculate $W_{2}$ from the original data and reject $H_{0}$ if $W_{2}>c_{2}$.\vskip .1in
\end{enumerate}
In our experiments, the performance of this second test can be on
a par with or, sometimes, even be slightly better than Test 1. In
our examples, this is particularly the case when random forest is
used as the classifier for estimating $p(x)$.

\section{Examples}

In this section we compare the performance of the classification probability
test (CPT) against maximum mean discrepancy (MMD) test \cite{Gretton-Borgwardt-al}
and the test based on classification accuracy (ACC) \cite{Lopez-Paz-Oquab,Ramdas-Singh-Wasserman}
. We choose the MMD test for comparison because the test is known
to have better overall performance among many commonly used two-sample
tests. We also include the ACC tests because of their similarities
to our tests. In our tests, the classifiers we use to estimate the
classification probabilities are random forests (RF) and support vector
machines (SVM), because they performed best among several popular
classifiers we tried in our experiments. The comparisons will be done
on simulated data from four different types of probability distributions
and on one real-life dataset. As a standard step in checking the validity
of a testing procedure, we first ran our test under the null hypothesis
$H_{0}:f=g$ with the datasets and confirmed that the test had the
correct type I error probability.

We use the receiver operating characteristic (ROC) curves and power
versus sample-size plots to compare the performance of the tests.
In the following, to obtain an approximation of the ROC curve for
a test, we first perform the test 400 times. Each time a new pair
of samples from the same pair of distributions is generated independently.
We collect all 400 $p$-values from these tests. The estimated ROC
curve is then the curve of the empirical distribution function of
these $p$-values. The curve allows us to determine the power of the
test at each given significance level. To obtain a plot of power versus
sample-size for a test, we set the significance level of the test
at $\alpha=0.05$ and then, for each of an increasing sequence of
sample sizes, we run the test independently 250 times to get a $p$-value
for that sample size. The power is then the proportion of these values
which are less than 0.05. In the following examples, each scatter
plot of the points of sample-size versus power has a smoothed curve
added, which is obtained from local polynomial regression fitting
to these points.

The same computations described above are applied to all examples
below. In the figures showing these examples, we write CPT-rf and
CPT-svm for the classification probability tests using RF and SVM
classifiers respectively, MMD for the maximum mean discrepancy test,
and ACC-rf and ACC-svm for the classification accuracy tests using
RF and SVM classifiers. We use function {\it ranger()} from the R
package {\it ranger} for RF classification, function {\it svm()}
from the R package {\it e1071} for SVM classification and function
{\it kmmd()} from the R package {\it kernlab} for MMD test. The
default values for the parameters in these functions are used. The
ACC-rf and ACC-svm tests use 2-fold cross-validation. All five tests
(CPT-rf, CPT-svm, ACC-rf, ACC-svm, MMD) in our examples below are
permutation tests. The null permutation distributions are all approximated
with 200 independent runs. 

\subsection{A comparison with the minimax rate for normal data}

First, we apply the tests in a standard setting to show how well it
performs against the known minimax power of the two-sample problem
\cite{Ramdas-Singh-Wasserman}. The first sample is from distribution
$N(0,\sigma^{2}I_{d\times d})$ and the second sample from $N(\delta,\sigma^{2}I_{d\times d})$
with $\delta\in\mathbb{R}^{d}$. Assume $n_{1}=n_{2}=n$. In this
case the minimax power over all tests for testing $\delta=0$ vs $\delta\not=0$
is known to be \cite{Ramdas-Singh-Wasserman}
\[
\phi(\alpha)=\Phi\left(\frac{\sqrt{d}}{\sqrt{d+n||\delta||_{2}^{2}/\sigma^{2}}}z_{\alpha}+\frac{||\delta||^{2}/\sigma^{2}}{\sqrt{8d/n^{2}+8||\delta||_{2}^{2}/n\sigma^{2}}}\right)+o(1),
\]
where $\Phi(\cdot)$ is the standard normal CDF. In our example, we
use $d=100$, $n=100$, $\sigma=2$ and try two different forms of
$\delta=(\delta_{1},\delta_{2},...,\delta_{100})$: a) a sparse case
with $\delta_{1}=1.6$ and $\delta_{i}=0$ for all $i>1$ and b) a
dense case with $\delta_{i}=0.16$ for all $i=1,...,100$. Both of
these forms of $\delta$ give the same $||\delta||_{2}=1.6$ and therefore
the same minimax lower bound.

\label{two-normals}
\begin{figure}[h]
\includegraphics[width=6in,height=3.1in]{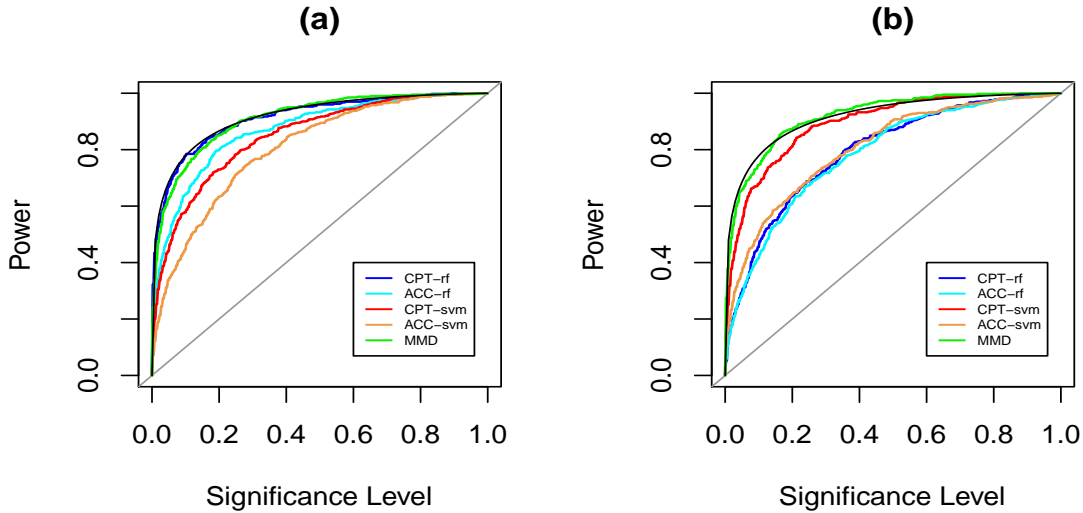}

\caption{The smooth black curves are the minimax power function $\phi(\alpha)$.
In (a) $\delta_{1}=1.6$, $\delta_{2}=\cdots=\delta_{100}=0$ and
in (b) $\delta_{1}=\cdots=\delta_{100}=0.16.$}
\end{figure}
Figure \ref{two-normals} displays the ROC curves of the tests for
both cases. The minimax power function $\phi(\alpha)$ is plotted
as black curves. We see that in (a) CPT-rf achieves the minimax bound
and in (b) CPT-svm gives a good approximation to the bound. We also
notice that the performance of MMD test are among the best in both
cases and that the default kernel used in MMD test is Gaussian. All
ACC tests are less powerful than their CPT counter parts. In the next
few examples, we will show how the CP tests outperform MMD in other
distribution settings.

\subsection{Distributions with the same means and variances but different covariances}

In this example, samples are drawn from two $d-$dimensional multivariate
normal distributions $N_{d}(\mu,\Sigma_{1})$ and $N_{d}(\mu,\Sigma_{2})$,
where the mean vector $\mu\in\mathbb{R}^{d}$ is identical in both
distributions, $\Sigma_{1}$ and $\Sigma_{2}$ are $d\times d$ covariance
matrices having the same diagonal elements but different off-diagonal
elements. Specifically, in this example, we set $d=100$, $\text{diag}(\Sigma_{1})=\text{diag}(\Sigma_{2})=(1.0,1.1,1.2,...,10.9)$,
$\mu_{1}=\mu_{2}=0_{d}$ and $\Sigma_{1}$ has all the off-diagonal
elements $0.01$ and $\Sigma_{2}$ has all the off-diagonal elements
0.21.

The performance of the tests on the data from these distributions
are shown in Figure \ref{diff-cov}. The ROC curves in (a) are based
on samples of size $m=n=100$ and the powers in (b) are for the tests
at the significance level 0.05. We see that CPT-svm test is the most
powerful and efficient test among the five tests.

\label{diff-cov}
\begin{figure}[H]
\includegraphics[width=6in,height=3.1in]{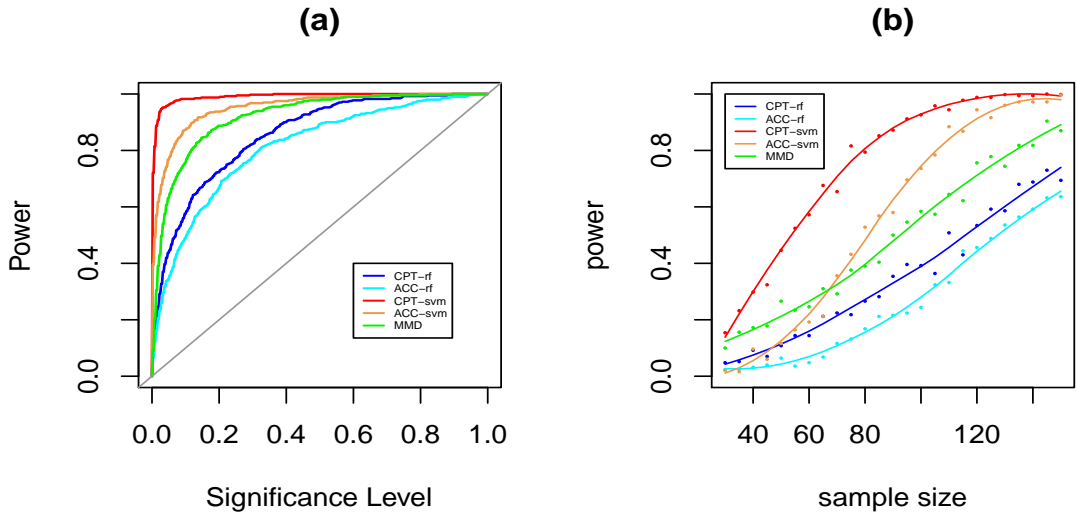}

\caption{Tests for distributions with the same mean vector but different covariance
matrices.}
\end{figure}

\subsection{Gaussian graphical models with different sparsities}

Here we have samples drawn from two Gaussian graphical models with
the same mean vectors but different precision matrices $Q_{1}$ and
$Q_{2}$ to reflect different degrees of connectivities of the underlying
graphs. More specifically, let $A_{1}$ be a weighted $d\times d$
symmetric adjacency matrix of a simple graph with its off-diagonal
entries randomly and independently taken from a uniform distribution
$U(0,1)$. Let $A_{2}$ be a more sparse adjacency matrix obtained
from $A_{1}$ by setting randomly and independently some of the entries
in $A_{1}$ to 0 according to a Bernoulli distribution $B(1,\tau)$.
Let $D_{1}$ and $D_{2}$ be the diagonal degree matrix of $A_{1}$
and $A_{2}$ respectively so that their $i$-th diagonal element is
the sum of the $i$-th row of their corresponding adjacency matrices.
Let $Q_{1}=(D_{1}-A_{1})+\delta_{1}I$ and $Q_{2}=(D_{2}-A_{2})+\delta_{2}I$,
where $\delta_{1},\delta_{2}>0$ are added to diagonals to make $Q_{1}$
and $Q_{2}$ nonsingular. In our experiment, means are the zero vector
$0_{200}$, $Q_{1}$ and $Q_{2}$ are 200 $\times$ 200 matrices (the
graphs have 200 nodes each) and $Q_{2}$ is obtained from $Q_{1}$
as described above with $\tau=0.65$.

\label{diff-sparsity}
\begin{figure}[H]
\includegraphics[width=6in,height=3.1in]{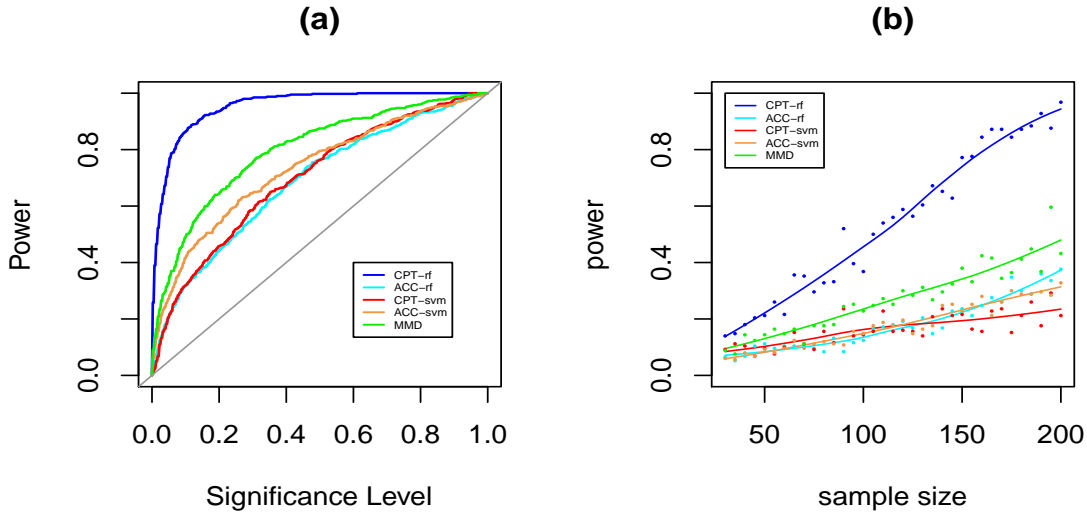}

\caption{Tests for different Gaussian graphical models with the same mean vector
but different rate of connectivities }
\end{figure}
The performance of the tests on data from these distributions are
shown in Figure \ref{diff-sparsity}. The ROC curves in (a) are based
on samples of sizes $n_{1}=n_{2}=150$ and the powers in (b) are for
the tests at the significance level 0.05. In this example, the CPT-rf
test is significantly more powerful and efficient than other tests.

\subsection{Distributions with the same means and covariances but one different
marginal distribution}

In this fourth example, sample one is from a $d$ dimensional distribution
which is a product measure of a one dimensional exponential distribution
with mean 1 and $d-1$ iid normal distributions of mean 1 and variance
1, $Exp(1)\times N_{d-1}(1_{d-1},I_{d-1})$. The second sample is
drawn from a product measure of $d$ iid normal distributions of mean
1 and variance 1, $N_{d}(1_{d},I_{d})$. Therefore both samples have
identical first and second moments and the only difference between
the two distributions is that one component in the first distribution
is exponential and the corresponding component in the second distribution
is normal. In the experiment, we use $d=100$.

\label{diff-margins}
\begin{figure}[h]
\includegraphics[width=6in,height=3.1in]{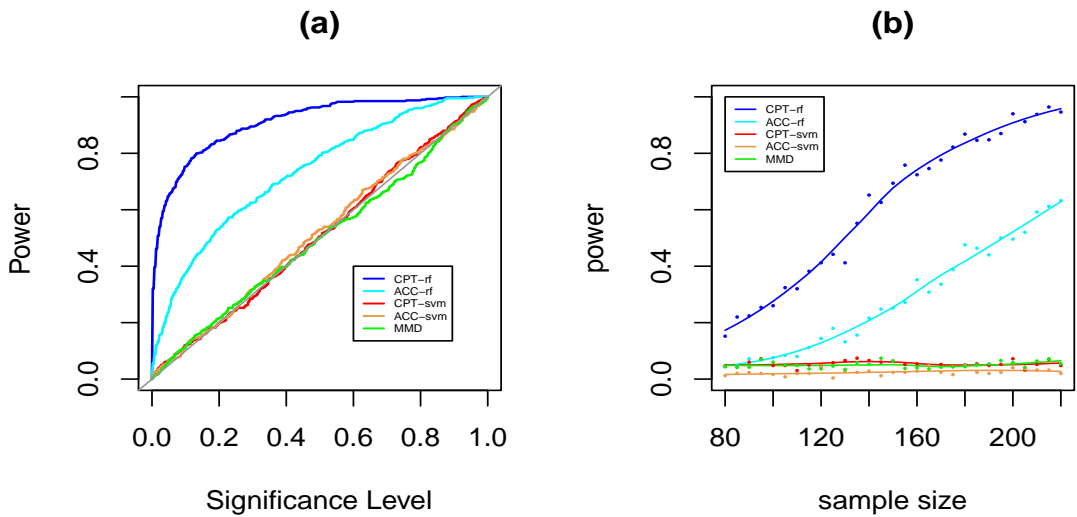}

\caption{Tests for distributions with the same means and covariances but one
different marginal distribution. }
\end{figure}
The performance of the tests on data from these distributions are
shown in Figure \ref{diff-margins}. The ROC curves in (a) are based
on samples of sizes $m=n=200$ and the powers in (b) are for the tests
at the significance level 0.05. We see that CAP-rf test outperforms
other tests in terms of the power and efficiency. In contrast, MMD
and SVM tests fail to detect any difference between the two distributions
from the samples. 

\subsection{Movie review sentiment data}

In this example, the samples of the same size $m=n$ are drawn from
two sets of movie reviews with 1000 reviews in each set. Of these
two sets of reviews, one has all the reviews with positive sentiments
and another all the reviews with negative sentiments. The null hypothesis
in the test is that there is no difference in sentiments between two
sets of reviews. All the reviews within each sample are encoded into
a so-called document-term matrix in which each review is represented
by a row of numbers 0 or 1 depending on wether a term which appears
in at least 5\% of all the reviews of the combined sample also appears
in this review. To make the data more challenging, in our experiment,
we purposely removed 50 highest influential words, including words
like ``bad'', ``worst'', ``dull'', ``excellent'', ``perfect''
etc. The dimension of the data depends on the samples we draw. It
varies around 1000.

The performance of the tests on these data are shown in Figure \ref{text-data}
The ROC curves in (a) are based on samples of sizes $m=n=75$ and
the powers in (b) are for the tests at the significance level 0.05.
It shows that CPT-rf test is best in terms of power and efficiency.

\label{text-data}
\begin{figure}[h]
\includegraphics[width=6in,height=3.1in]{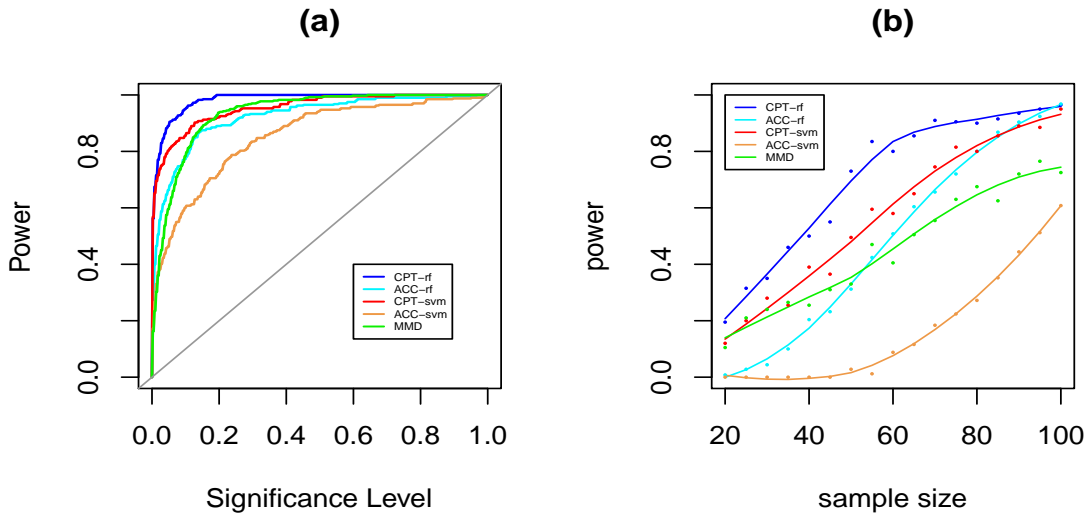}

\caption{Tests for difference of text files with two sentiments. }
\end{figure}

\subsection{Comparing Classification Probability Test 1 and Test 2}

In this last example, we use the data from the first four examples
above to show the difference between Test 1 and Test 2 based on classification
probabilities. It shows that if SVM is used as the classifier, then
Test 1 is consistently better than Test 2. But if the classifier is
RF, then Test 2 can often be better.

\label{test1-test2}
\begin{figure}[h]
\includegraphics[width=6in,height=6in]{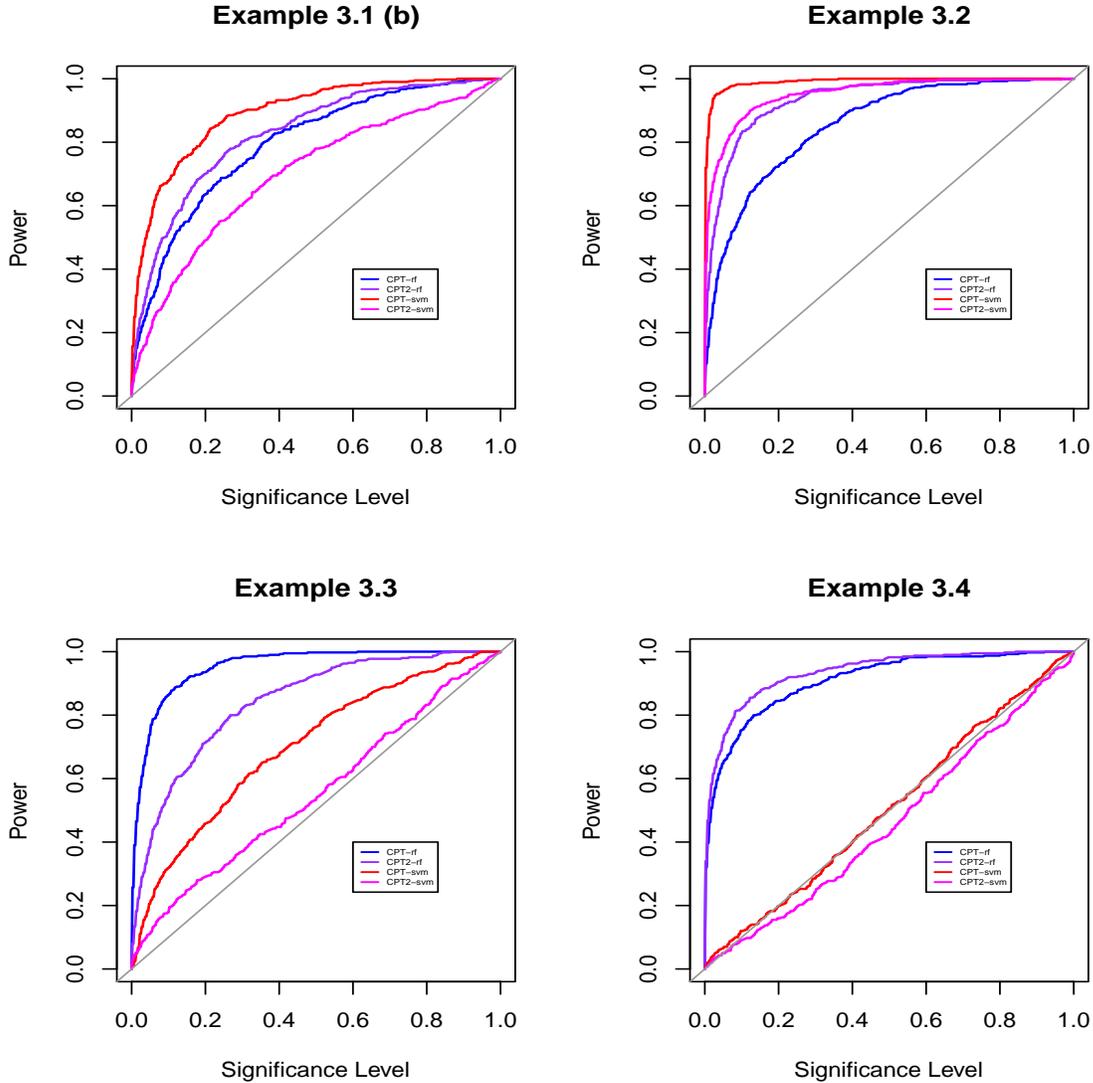}

\caption{ROC curves for CP tests using $W_{1}$ and $W_{2}$ on the datasets
from the first four examples. }
\end{figure}

\section{Summary and discussion}

We propose a test for two-sample problem based on estimates of classification
probabilities obtained from a consistent classification algorithm.
This test is effectively a likelihood-ratio-based test in which the
ratio is estimated not by maximizing the likelihoods involved - which
would require knowledge on the density functions, but by estimating
the odds ratio of classification probabilities. Our test is more powerful
and efficient than many other tests.

Our test also has taken advantage of the recent advances in classification
algorithms and computing software to derive its performance in computation.
For any given dataset, tests using different classifiers deliver different
performance, depending on the structure of underlying probability
distributions. We conjecture that this phenomenon is due to the quality
of estimates of the classification probability when using different
classifiers on different types of data. More studies are needed to
better understand the statistical properties of these estimates. In
practice, one can choose a best classifier using K-fold cross-validation
to ensure the best performance.

Because these are permutation tests, they can be computationally intensive.
To get some idea on CPU time usage in a single test, we provide the
following table, in which each entry for time is the elapsed time
in seconds returned from the function {\it system.time()} in R running
on a MacBook Pro with 2.2 GHz Intel Core i7 processor and macOS Mojave
version 10.14.5. In all these tests, every null permutation distribution
was approximated using 200 independent runs.

\begin{table}[H]
\caption{Elapsed Time of the Tests}

\vspace{0.2in}

\begin{tabular}{|c|c|c|c|c|c|c|c|}
\hline 
 & $d$ & $n$ & CPT-svm & ACC-svm & CPT-rf & ACC-rf & MMD\tabularnewline
\hline 
\hline 
Example 1 & 100 & 100 & 12.143 & 8.221 & 15.532 & 14.293 & 2.408\tabularnewline
\hline 
Example 5 & 1000 & 75 & 63.772 & 48.820 & 30.772 & 43.555 & 4.373\tabularnewline
\hline 
\end{tabular}
\end{table}


\begin{thebibliography}{10}
\bibitem{Bai-Saranadasa}Z. Bai and H. Saranadasa, Effect of High
Dimension: by An Example of A Two Sample Problem. Statist. Sin., 6,
(1996), 311\textendash 329.

\bibitem{Biau-Gyorfi} G. Biau and L. Gyorfi, On the Asymptotic Properties
of a Nonparametric L1-test Statistic of Homogeneity, IEEE Transactions
on Information Theory, Vol 51, 11 (2005), 3965-3973.

\bibitem{Cai-Liu-Xia} Tony Cai, Weidong Liu, and Yin Xia, Two-Sample
Covariance Matrix Testing and Support Recovery in High-Dimensional
and Sparse Settings, Journal of American Statistical Association,
Vol 108, No. 501, 2013, 256-277.

\bibitem{Chen-Qin}S. X. Chen and Y. L. Qin, A Two-sample Test for
High-Dimensional Data with Applications to Gene-set Testing. Ann Statist,
38 (2010), 808\textendash 835.

\bibitem{Chakraborty-Chaudhuri} A. Chakraborty and P. Chaudhuri,
A Wilcoxon-Mann-Whitney Type Test for Infinite Dimensional Data, Biometrika,
Vol 102, Issue 1 (2015), 239\textendash 246.

\bibitem{Chwialkowski-al.} K. Chwialkowski, A. Ramdas, D. Sejdinovic
and A. Gretton, Fast Two-Sample Testing with Analytic Representations
of Probability Measures, NIPS'15 Proceedings of the 28th International
Conference on Neural Information Processing Systems - Vol 2, 1981-1989,
2015.

\bibitem{Friedman} J. H. Friedman, On Multivariate Goodness-of-fit
and Two-sample Testing. Technical Report SLAC-PUB-10325, Stanford
Linear Accelerator Center, Stanford University, 2004.

\bibitem{Hardle-Serfling} P. Hardle and R. Serfling, Strong Uniform
Consistency Rates for Estimators of Conditional Functions, The Annals
of Statistics, Vol. 16 No. 4 (1988), 1428-1449.

\bibitem{Hu-Bai} J. Hu and Z. Bai, A Review of 20 Years of Naive
Tests of Significance for High-dimensional Mean Vectors and Covariance
Matrices, Science China Mathematics, Vol 59, 12 (2016), 2281-2300.

\bibitem{Gine-Guillou}E. Gine and A. Guillou, Rates of Strong Uniform
Consistency for Multivariate Kernel Density Estimators. Ann. Inst.
H. Poincaré Probab. Statist. 38 (2002), 907\textendash 921.

\bibitem{GyorfiL} L. Gyorfi, Michael Kohler, Adam Krzyzak and Harro
Walk, A Distribution-Free Theory of Nonparametric Regression, Springer-Verlag
New York, 2002.

\bibitem{Gretton-Borgwardt-al} A. Gretton, K. M. Borgwardt, M. J.
Rasch, B. Scholkopf, and A. Smola, A Kernel Two-Sample Test, Journal
of Machine Learning Research, 13 (2012), 723-773.

\bibitem{Kruppa-et-al} J. Kruppa1, Y. Liu, H. C. Diener, T. Holste1,
C. Weimar, I. R. K\textasciidieresis onig1, and A. Ziegler, Probability
Estimation with Machine Learning Methods for Dichotomous and Multicategory
Outcome: Applications, Biometrical Journal 56, 4 (2014), 564\textendash 583.

\bibitem{Lopez-Paz-Oquab} D. Lopez-Paz and M. Oquab, Revisiting Classifier
Two-Sample Tests, International Conference on Learning Representations
(ICLR) 2017.

\bibitem{Malley-el-al} J. D. Malley, J. Kruppa, A. Dasgupta, K. G.
Malley, and A. Ziegler, Probability Machines: Consistent Probability
Estimation Using Nonparametric Learning Machines. Methods of Information
in Medicine 51 (2012).

\bibitem{Mondal-Biswas-Ghosh} P. K. Mondal, M. Biswas, and A. K.
Ghosh, On High Dimensional Two-Sample Tests Based on Nearest Neighbors,
Journal of Multivariate Analysis, 141 (2015), 168-178.

\bibitem{Nenze} N. Henze, A Multivariate Two-Sample Test Based on
the Number of Nearest Neighbor Type Coincidences, The Annals of Statistics,
Vol 16, 2 (1988), 772-783.

\bibitem{Ramdas-Singh-Wasserman} A. Ramdas, A. Singh and L. Wasserman,
Classification Accuracy as a Proxy for Two Sample Testing. arXiv preprint
arXiv:1602.02210, 2016.

\bibitem{Stadler-Mukherjee} N. Stadler and S. Mukherjee, Two-Sample
Testing in High Dimensions, J. R. Statist. Soc. B, Vol 79, Issue 1
(2017), 225-246.

\bibitem{Steinwart} I. Steinwart, Support vector machines are universally
consistent, Journal of Complexity 18 (2002), 768\textendash 791.

\bibitem{Steinwart-Christmann} I. Steinwart and A. Christmann, Support
Vector Machines, Springer Science+Business Media, LLC, 2008.

\bibitem{Wager-Athey} S. Wager and S. Athey, Estimation and Inference
of Heterogeneous Treatment Effects using Random Forests, arXiv, 2017.

\bibitem{Wager-Hastie-Efron} S. Wager, T. Hastie and B. Efron, Confidence
Intervals for Random Forests: The Jackknife and the Infinitesimal
Jackknife, Journal of Machine Learning Research 15 (2014), 1625-1651.
\end{thebibliography}
\end{document}